\documentclass[a4paper, 11pt]{article}
\usepackage[utf8]{inputenc}
\usepackage{verbatim}
\usepackage[margin=3cm]{geometry}
\usepackage[algo2e,ruled,vlined]{algorithm2e}
\usepackage{algorithm}
\usepackage{here}
\usepackage{comment}
\usepackage{setspace}
\usepackage[normalem]{ulem}
\usepackage{algpseudocode}
\usepackage{comment}
\usepackage[T1]{fontenc}
\usepackage{dsfont}
\usepackage{url}
\usepackage[inline]{enumitem}
\usepackage{microtype}
\usepackage[affil-sl]{authblk}
\usepackage{tikz}
\usepackage{amsmath}
\usepackage{amssymb}
\usepackage{amsthm}
\usepackage{mathtools}
\usepackage{nicefrac}
\usepackage{amsfonts,latexsym,graphics}
\usepackage{color}
\usepackage[hidelinks]{hyperref}
\usepackage{array,multirow}
\usepackage[english]{babel}
\usepackage{epsfig}
\usepackage{graphicx}
\usepackage{bm}
\usepackage{makeidx}
\usepackage{xspace}

\newtheorem{thm}{Theorem}
\newtheorem{proposition}[thm]{Proposition}

\newtheorem{cor}[thm]{Corollary}
\newtheorem{lemma}[thm]{Lemma}
\newtheorem{remark}[thm]{Remark}
\newtheorem{example}[thm]{Example}
\newtheorem{question}[thm]{Question}
\newtheorem{problem}[thm]{Problem}
\newtheorem{con}[thm]{Conjecture}

\def\L{{\cal L}}
\def\LL{{\sf L}}

\def\1{{\sf 1}}

\def\ccap{{\sf cap}}
\begin{document}

\title{The $M$-matrix group inverse problem\\
for distance-biregular graphs}

\author{Aida Abiad\thanks{\texttt{a.abiad.monge@tue.nl},  Department of Mathematics and Computer Science, Eindhoven University of Technology, The Netherlands\\Department of Mathematics: Analysis, Logic and Discrete Mathematics, Ghent University, Belgium\\Department of Mathematics and Data Science of Vrije Universiteit Brussel, Belgium}  \quad \'Angeles Carmona\thanks{\texttt{angeles.carmona@upc.edu},  Department of Mathematics, Polytechnic University of Catalonia, Spain}
\quad Andr\'es M. Encinas\thanks{\texttt{andres.marcos.encinas@upc.edu}, Department of Mathematics, Polytechnic University of Catalonia, Spain} \quad Mar\'\i a Jos\'e Jim\'enez\thanks{\texttt{maria.jose.jimenez@upc.edu}, Department of Mathematics, Polytechnic University of Catalonia, Spain}}

\date{}

\maketitle

\begin{abstract}
In this paper we provide the group inverse of the combinatorial Laplacian matrix of distance--biregular graphs using the so--called equilibrium measures for sets obtained by deleting a vertex. We also show that the two equilibrium arrays characterizing a distance--biregular graph can be expressed in terms of the mentioned equilibrium measures. As a consequence of the minimum principle, we show a characterization of when the group inverse of the combinatorial Laplacian matrix of a distance--biregular graph is an $M$--matrix. 
\end{abstract}


{\bf AMS Classification:} 15A09, 05C50, 51E05, 51E12

\section{Introduction}

One problem with the theory of distance-regular graphs is that it does not apply directly to the graphs of generalised polygons. Godsil and Shawe--Taylor \cite[]{gs1987}  overcame this difficulty by introducing the class of distance-regularised graphs, a natural common generalisation. These graphs are shown to either be distance-regular or distance-biregular. This family includes the generalised polygons and other interesting graphs. Distance-biregular graphs, which were introduced by Delorme \cite{d1983} in 1983, can be viewed as a bipartite variant of distance-regular graph: the graphs are bipartite and for each vertex there exists an intersection array depending on the stable component of the vertex. Thus such graphs are to distance-regular graphs as bipartite regular graphs are to regular graphs. They also are to non-symmetric association schemes as distance-regular graphs are to symmetric association schemes. Since their introduction, distance-biregular graphs have received quite some attention, see \cite{A1990,  c1999part1,c1999part2,  d1994,  F2013,HP2021, MS1985} or \cite[Chapter 4]{BCN89} for an overview.

In the first part of this paper we obtain the  group inverse of the combinatorial Laplacian matrix of distance-biregular graphs. The group inverse matrix can be seen in the framework of discrete potential theory as the Green's functions associated with the Laplacian operator and it can be used to deal with diffusion-type problems on graphs, such as chip-firing, load balancing, and discrete Markov chains. For some graph classes, the group inverse is known. Instances of it are the work of Urakawa \cite{U1997}, Bendito et al. \cite{BCE00,BCEM10-1} or more recently the study of the Green function for forests by Chung and Zeng \cite{cz2021}. Other generalized inverses, such as the Moore-Penrose inverse, have been studied. For instance, the Moore-Penrose inverse of the  incidence matrix of several graphs has been investigated by Azami and Bapat \cite{AB2019v1,AB2019v2,AB2018}. Nevertheless, the problem of computing group inverses still remains wide open for most graph classes. We show an explicit expression for  the group inverse of the combinatorial Laplacian matrix of a distance-biregular graph in terms of its intersection numbers. This result, together with the group inverse  of a distance-regular graph found by Bendito, Carmona and Encinas \cite{BCEM10-1}, and independently, by Chung and Yau \cite{fy2000}, completes the corresponding study for distance-regularised graphs.

In matrix theory, the Laplacian matrix is known to be a symmetric $M$-matrix (a symmetric positive semi-definite matrix with non-positive off-diagonal elements). Nonnegative matrices and $M$-matrices have become a staple in contemporary linear algebra, and they arise frequently in its applications. Such matrices are encountered not only in matrix analysis, but also in stochastic processes, graph theory, electrical networks, and demographic models \cite{KN95}.  A fundamental problem related with $M$-matrices is the so-called inverse $M$-matrix problem, that consists in characterizing all non-negative matrices whose inverses are $M$-matrices. For singular matrices, the inverse problem was originally posted by Neumann, Poole and Werner as follows:

\begin{question}\cite{DN1984,NPW1982}\cite[Question 3.3.8]{kirklandbook}\label{que:Mmatrix}
Characterize all singular and irreducible $M$-matrices for which its group inverse is also an $M$-matrix.
\end{question}

This question has only been answered for specific matrix classes. In the graph setting, this question has been solved for weighted trees by Kirkland and Neumann \cite{KN98}, and for distance-regular graphs by Bendito, Carmona and Encinas \cite{BCEM2012}. In a more general setting, Question \ref{que:Mmatrix} has been investigated for nonnegative matrices having few eigenvalues by Kirkland and Neumann \cite{KN95}, for periodic and nonperiodic Jacobi matrices by Chen, Kirkland and Neumann \cite{CKN95} and for general symmetric $M$-matrices whose underlying graphs are paths by Bendito, Carmona and Encinas \cite{BCE2012} and Carmona, Encinas and Mitjana \cite{CEM2013}. Recently, matrices whose group inverses are $M$-matrices were investigated by Kalauch, Lavanya and Sivakumar \cite{KLS2021}. 

We answer Question \ref{que:Mmatrix} for distance-biregular graphs, completing, together with the known results for distance-regular graphs \cite{BCEM2012}, the characterization of when the group inverse of the combinatorial Laplacian matrix of a distance--regularised  graph is an $M$--matrix.
\section{Preliminaries}
The triple  $\Gamma=(V,E,c)$ denotes a finite network; that is, a
finite connected graph without loops or multiple edges, with vertex
set $V$, whose cardinality equals $n\ge 2$, and edge set $E$, in which
each edge $\{x,y\}$ has been assigned a {\it conductance} $c(x,y)
>0$. The conductance can be considered as a symmetric function
$c\colon V\times V\longrightarrow [0,+\infty)$ such that $c(x,x)=0$
for any $x\in V$ and moreover, $x\sim y$, that is  vertex $x$ is
adjacent to vertex $y$, iff $c(x,y)
>0$. We define the {\it degree function} $k$ as $$k(x)=\sum\limits_{y\in
V}c(x,y)$$ for each $x\in V$. The usual distance from vertex $x$ to
vertex $y$ is denoted by $d(x, y)$ and $D = \max\{d(x, y) : x, y \in
V \}$ stands for the {\it diameter} of $\Gamma$. We denote as
$\Gamma_i(x)$ the set of vertices at distance $i$ from vertex $x$,
$\Gamma_i(x)=\{y:d(x,y)=i\},\, 0\le i\le D$ and define $k_i(x)= \big |\Gamma_i(x)\big|$. Then, $$B_i(x)=\sum\limits_{j=0}^i k_j(x)$$ is  the cardinal of the $i$-ball centered at $x$.  The {\it complement} of $\Gamma$ is defined as the graph $\overline{\Gamma}$ on the same
vertices such that two vertices are adjacent iff they are not
adjacent in $\Gamma$; that is $x\sim y$ in $\overline{\Gamma}$ iff
$c(x,y)=0$. More generally, for any $i=1,\ldots,D$, we denote by
$\Gamma_i$ the graph whose vertices are those of $\Gamma$ and in
which two vertices are adjacent iff they are  at distance $i$ in
$\Gamma$. Therefore for any $x\in V$, $\Gamma_i(x)$ is the set of
adjacent vertices to $x$ in $\Gamma_i$. Clearly $\Gamma_1$ is the
graph subjacent to the network $\Gamma$ and
$\Gamma_2=\overline{\Gamma}$ when $D=2$. 

The set of real-valued functions on $V$ is denoted by ${\cal
C}(V)$. When necessary, we identify the functions in  ${\cal C}(V)$
with vectors in   $\mathbb{R}^{|V|}$ and the endomorphisms of ${\cal
C}(V)$ with $|V|$-order square matrices.

The {\it combinatorial Laplacian} or simply the {\it Laplacian} of
the graph $\Gamma$ is the endomorphism of
${\cal C}(V)$  that assigns to each $u\in{\cal C}(V)$ the function
\begin{equation}
\label{laplacian}{\cal L}(u)(x)\displaystyle=\sum\limits_{y\in V}
c(x,y)\,\Big(u(x)-u(y)\Big)=k(x)u(x)-\sum\limits_{y\in V}
c(x,y)\,u(y),\hspace{.25cm}x\in V.
\end{equation}

It is well-known that ${\cal L}$ is a positive semi-definite
self-adjoint operator and has $0$ as its lowest eigenvalue whose
associated eigenfunctions are constant. So, ${\cal L}$ can be
interpreted as an irreducible, symmetric, diagonally dominant and
singular $M$-matrix, that in the sequel will be denoted as ${\sf L}$. Therefore, the Poisson equation
${\cal L}(u)=f$ on $V$ has solution iff $$\sum\limits_{x\in V}f(x)=0$$
and, when this happens, there exists a unique solution $u\in{\cal
C}(V)$ such that $\sum\limits_{x\in V}u(x)=0,$ see \cite{BCE00}.

The {\it Green operator} is the linear operator ${\cal G}:{\cal
C}(V)\longrightarrow {\cal C}(V)$ that assigns to any $f\in{\cal
C}(V)$ the unique solution of the Poisson equation 
${\cal L}(u)=f-\frac{1}{n}\sum\limits_{x\in V}f(x)$ such that $\sum\limits_{x\in
V}u(x)=0$. It is easy to prove that ${\cal G}$ is a positive
semi-definite  self-adjoint operator and has $0$ as its lowest
eigenvalue whose  associated eigenfunctions are constant. Moreover,
if ${\cal P}$ denotes the projection on the subspace of constant
functions then,
$${\cal L}\circ {\cal G}={\cal G}\circ{\cal L}={\cal I}-{\cal P}.$$
In addition, we define the {\it Green function} as $G:V\times
V\longrightarrow\mathbb{R}$ given by $G(x,y)={\cal
G}(\varepsilon_y)(x)$, where $\varepsilon_y$ stands for the Dirac
function at $y$. Therefore, interpreting ${\cal G}$, or $G$, as a
matrix  it is nothing else but ${\sf L}^\#$ the group inverse
inverse of ${\sf L}$, that coincides with its Moore-Penrose inverse. In
consequence, ${\sf L^\#}$ is a  \emph{$M$-matrix} iff ${\sf L^\#}(x,y)\le 0$ for
any $x,y\in V$ with $x\ne y$ and then ${\sf L^\#}$ can be identified with
the combinatorial Laplacian matrix of a new connected network with the same vertex set, that we denote by $\Gamma^\#$.

From now on we will say that a network $\Gamma$ has the {\it $M$-property} iff ${\sf L}^\#$ is an $M$-matrix; that is, if ${\sf L}$ provides an answer to Question \ref{que:Mmatrix}.

In \cite{BCE00} it was proved that for any $y\in V$,  there exists a unique
$\nu^y\in {\cal C} (V)$ such that $\nu^y(y)=0$,
$\nu^y(x)>0$ for any $x\not=y$ and satisfying
\begin{equation}
\label{L-equilibrio}
{\cal L}(\nu^y)=\1-n\varepsilon_y\hspace{.25cm}\hbox{on
$V$}.
\end{equation}
We call $\nu^y$ the {\it equilibrium measure of
$V\setminus\{y\}$} and then we define {\it capacity} as  the function $\ccap\in
{\cal C} (V)$ given by $\ccap(y)=\sum\limits_{x\in V}
\nu^y(x)$.

Following the ideas in \cite{BCE00,BCEM2012,U1997}, we define, for any $y\in V$, the {\it equilibrium array for $y$} as the set $\{\nu^y(x):x\in V\}$ of different values taken by the equilibrium measure of $y$, and we consider the {\it length of the equilibrium array} to be $\ell(y)=\big|\big\{\nu^y(x):x\in V\setminus\{y\}\big\}\big|$. Since $\Gamma$ is connected and $n\ge 2$, we obtain that $\ell(y)\ge 1$ for any $y\in V$. On the other hand, since $0=\nu^y(y)$ we obtain that   $\{\nu^y(x):x\in V\}=\{q_i(y): i=0,\ldots,\ell(y)\}$, where $0=q_0(y)<q_1(y)<\cdots<q_{\ell(y)}(y)$.  In addition, given $y\in V$ for any $i=0,\ldots,\ell(y)$, we define $m_i(y)=\big |\{x\in V: \nu^y(x)=q_i(y)\}\big|$. Clearly, for any $y\in V$ we have that

$$
 m_0(y)=1, \,\,\,
   n=\sum\limits_{i=0}^{\ell(y)}m_i(y),\,\,\,\mbox{and }
   \ccap(y)=\sum\limits_{i=1}^{\ell(y)}m_i(y)q_i(y).
$$


In \cite[Proposition 3.12]{BCEM2000shortest} it was shown that, for any $y\in V$, the equilibrium measure (and hence the equilibrium array) reflects the graph depth from $y$, since 
\begin{equation}
\label{distance}
\nu^y(x)=q_i(y) \Longrightarrow d(x,y)\le i
\end{equation}

\noindent and hence $$\sum\limits_{j=0}^i m_j(y)\le B_i(y)$$ for any $0\le i\le D$. In particular, \eqref{distance} implies that if $\nu^y(x)=q_1(y)$ then $x\sim y$; that is, that the minimum values of the equilibrium measure for $y$ are attained at  vertices adjacent to $y$ (in fact this a formulation of the so-called {\it minimum principle}). 

In general, when  $d(x,y)=i$, Property \eqref{distance}  only assures that $\nu^y(x)\ge q_i(y)$, but the inequality can be strict. In particular the length of some equilibrium arrays could be greater than $D$.

\begin{example}\label{ex:cycle}
To illustrate the above statements, consider the complete graph $K_3$ with vertex set $V=\{x_1,x_2,x_3\}$ and conductances $c_1=c(x_1,x_2)$, $c_2=c(x_2,x_3)$ and $c_3=c(x_3,x_1)$. Then, 

{\small{
\[\begin{array}{rlrlrl}
\nu^{x_1}(x_2)=&\hspace{-.25cm}\dfrac{2c_2+c_3}{c_1c_2+c_2c_3+c_3c_1}, &\hspace{-.25cm} \nu^{x_1}(x_3)=&\hspace{-.25cm}\dfrac{2c_2+c_1}{c_1c_2+c_2c_3+c_3c_1}, &\hspace{-.25cm}\ccap(x_1)=&\hspace{-.25cm}\dfrac{4c_2+c_1+c_3}{c_1c_2+c_2c_3+c_3c_1}\\[3ex]
\nu^{x_2}(x_1)=&\hspace{-.25cm}\dfrac{2c_3+c_2}{c_1c_2+c_2c_3+c_3c_1}, &\hspace{-.25cm} \nu^{x_2}(x_3)=&\hspace{-.25cm}\dfrac{2c_3+c_1}{c_1c_2+c_2c_3+c_3c_1}, &\hspace{-.25cm}\ccap(x_1)=&\hspace{-.25cm}\dfrac{4c_3+c_1+c_2}{c_1c_2+c_2c_3+c_3c_1}\\[3ex]
\nu^{x_3}(x_1)=&\hspace{-.25cm}\dfrac{2c_1+c_2}{c_1c_2+c_2c_3+c_3c_1}, & \hspace{-.25cm}\nu^{x_3}(x_2)=&\hspace{-.25cm}\dfrac{2c_1+c_3}{c_1c_2+c_2c_3+c_3c_1},&\hspace{-.25cm}\ccap(x_1)=&\hspace{-.25cm}\dfrac{4c_1+c_2+c_3}{c_1c_2+c_2c_3+c_3c_1}.
\end{array}
\]
}}

So, $D=1$, but $\ell(x_1)=1$ iff $c_1=c_3$, $\ell(x_2)=1$ iff $c_1=c_2$ and $\ell(x_3)=1$ iff $c_2=c_3$.

\end{example}

The group inverse of the Laplacian matrix and the equilibrium measures provide an equivalent information about the network structure, since the expression of $\LL^\#$ can  be obtained from  equilibrium measures  and conversely. Specifically, see \cite[Proposition 3.9]{BCE00}, the group inverse $\LL^\#$  is given by
\begin{equation}
\label{green-f}
\LL^\#(x,y)=\dfrac{1}{n^2}\big(\ccap(y)-n\,\nu^y(x)\big)
\end{equation}
and this equality also implies that $\ccap(y)=n^2\LL^\#(y,y)$ and that 
\begin{equation}
\label{equil-green}
\nu^y(x)=n\,\big(\LL^\#(y,y)-\LL^\#(x,y)\big),\hspace{.25cm}x,y\in V.
\end{equation}
In addition, the symmetry of the group inverse  leads to the following relation for the equilibrium measures 
\begin{equation}
\label{equil-sym}
\nu^y(x)-\nu^x(y)=\frac{1}{n}\big(\ccap(y)-\ccap(x)\big)=n\,\big(\LL^\#(y,y)-\LL^\#(x,x)\big),\hspace{.25cm}x,y\in V.
\end{equation}

From (\ref{green-f}) the minimum principle states that
a network $\Gamma$ has the $M$-property iff for any $y\in V$ 
\begin{equation}
\label{measure-m}
\ccap(y)\le n\nu^y(x) \quad \hbox{ for any } x \sim y,
\end{equation}
see \cite[Theorem 1]{BCEM2012}. In this case,  $\overline{\Gamma}$ is a subgraph of the subjacent graph of $\Gamma^\#$.  %
In fact, to achieve the $M$-property it is sufficient to  satisfy that 
$$\sum\limits_{i=1}^{\ell(y)}m_i(y)q_i(y)\le nq_1(y)$$
for any $y\in V$. Since this inequality trivially holds when $\ell(y)=1$, and assuming the common agreement that empty sum equals $0$, we have that $\Gamma$ has the $M$-property iff 
\begin{equation}
\label{equil-m}
\sum\limits_{i=2}^{\ell(y)}m_i(y)\big(q_i(y)-q_1(y)\big)\le q_1(y)
\end{equation}
for any $y\in V$. Therefore, when $\ell(y)=1$ for any $y\in V$, then $\Gamma$ is a complete network and moreover satisfies the $M$-property. As Example \ref{ex:cycle} shows, a complete network does not necessarily satisfy the $M$-property: $K_3$ has the $M$-property if and only if $3\max\{c_1,c_2,c_3\}\le 2(c_1+c_2+c_3)$. In particular, if $c_1=c_2=c_3$, then $K_3$ has the $M$-property, but  if  for instance $c_3>2(c_1+c_2)$, then $K_3$ does  not satisfy the $M$-property.

\section{Group inverse for distance-biregular graphs}\label{sec:greenfunction}

 We say that the graph $\Gamma=(V,E)$ is \emph{semiregular} if $\Gamma$ is bipartite with $V=V_0 \cup V_1$ and there are numbers $k_0$ and $k_1$ such that each vertex in $V_0$ has $k_0$ neighbors and each vertex in $V_1$ has $k_1$ neighbors. In this case, we define $D_\ell=\max \{d(x,y):\,\, y\in V, x\in V_\ell\}$, $\ell=0,1$. Moreover, for $\ell=0,1$, we denote by $\bar\ell=1-\ell.$ In the sequel without loss of generality we always suppose that  $1\le D_0\le D_1$.

A connected graph $\Gamma$ is a  \emph{distance-biregular graph} if  $\Gamma$ is semiregular and for any two vertices $x$ and $y$ at distance $i$, the numbers $|\Gamma_{i-1}(x) \cap \Gamma_{1}(y)|$ and $|\Gamma_{i+1}(x) \cap \Gamma_{1}(y)|$ only depend on $i$ and on the stable set where $x$ is.

Examples of distance-biregular graphs are the subdivision graph of minimal $(k,g)$-cages. In particular, the subdivision graph of the Petersen graph is a distance-biregular graph. Also, any bipartite distance-regular graph is a distance-biregular graph with $k_0=k_1$.

For $x\in V_\ell,$ $\ell=0,1$, we define the {\it intersection numbers} by $c_{\ell,i}=|\Gamma_{i-1}(x) \cap \Gamma_{1}(y)|$ and $b_{\ell,i}=|\Gamma_{i+1}(x) \cap \Gamma_{1}(y)|$, $i=0,\ldots,D_\ell$, with the usual agreement $c_{\ell,0}=b_{\ell,D_\ell}=0$. Clearly, for $\ell=0,1$ it is satisfied that $b_{\ell,0}=k_\ell$, $c_{\ell,1}=1$, $b_{\ell,1}=k_{\bar \ell}-1$, $\ell=0,1$ and more generally for any $i\in \{0,\ldots, D_\ell \}$ the following holds
\begin{align*}
 c_{\ell,i}+b_{\ell,i}=\left\{ \begin{array}{cc} 
                k_\ell & \text{if } i \text{ is even,}\\
                k_{\bar\ell} &  \text{if } i \text{ is odd.} \\
                \end{array}   \right.
\end{align*}
Therefore, a distance-biregular graph has a double intersection array which will be denoted by
\begin{align*}
 \big\{ k_\ell;  c_{\ell,1},  \ldots  c_{\ell,D_\ell} 
               \big\}, \,\,\,\ell=0,1.
\end{align*}
If, for $i\in \{0,\ldots, D_\ell \}$ and $x\in V_\ell$ we define $k_{\ell,i}=|\Gamma_{i}(x)|$ and $B_{\ell,i}=\sum\limits_{j=0}^ik_{\ell,j}$. Then, $k_{\ell,0}=1$,  $k_{\ell,1}=k_{\ell}$ and $n=B_{\ell,D_\ell}$, $\ell=0,1$ and moreover the following  relationships hold, see the Lemmas 2.1-2.8 in \cite[Section 2.1]{A1990} and the  references therein.
\begin{lemma} \label{propertiess}If $\Gamma$ is  a distance-biregular graphs with intersection arrays $\big\{ k_\ell;  c_{\ell,1},  \ldots  c_{\ell,D_\ell} 
               \big\}$, $\ell=0,1$. Then,
        
\begin{itemize}
\item[$(i)$] $0\le D_1-D_0\le 1$ and when $D_1=D_0+1$, then $D_0$ is odd.
\item[$(ii)$] For $\ell=0,1$, $$\displaystyle k_{\ell,i}=\prod_{j=0}^{i-1} \frac{b_{\ell,j}}{c_{\ell,j+1}},\quad i=0,\ldots,D_\ell$$

and hence, $$k_{\ell,i}b_{\ell,i}=k_{\ell,i+1}c_{\ell,i+1},\quad  i=0,\ldots,D_\ell.$$
\item[$(iii)$] $k_0k_{1,2i+1}=k_1k_{0,2i+1}$, for any $i=0,\ldots,\lfloor\frac{D_0-1}{2}\rfloor$.  
\item[$(iv)$] $c_{0,2i}c_{0,2i+1}=c_{1,2i}c_{1,2i+1}$ and  $b_{0,2i-1}b_{0,2i}=b_{1,2i-1}b_{1,2i}$ for any $i=1,\ldots,\lfloor\frac{D_0-1}{2}\rfloor$.
\item[$(v)$] For $\ell=0,1$, $1\le c_{\ell,i}\leq c_{\bar\ell,i+1}$ and $b_{\ell,i}\geq  b_{\bar\ell,i+1}$ for any $i=0,\ldots,D_{\bar\ell}-1$. Moreover,  $b_{\ell,i}\geq  c_{\bar\ell,i+1},\hspace{.25cm} i=1,\ldots,D_{\bar\ell}-2$. 
%
\item[$(vi)$] For $\ell=0,1$, $c_{\ell,2}\le {c_{ \ell,3}-1\choose c_{\bar\ell,2}-1}$. 
 \item[$(vii)$]  For $\ell=0,1$, if $i+j$ is even and $i+j\le D_\ell$, then $c_{\ell,i}\leq b_{\ell,j}.$
    \item [$(viii)$] For $\ell=0,1$, if $i+j$ is odd and $i+j\le  D_0$, then $c_{\ell,i}\leq b_{\bar\ell,j}$ and $c_{\bar\ell,i}\leq b_{\ell,j}$.
\end{itemize}
\end{lemma}
The properties $(ii)$, $(iii)$ and $(iv)$ imply that for $\ell=0,1$, the intersection numbers $\{c_{\ell,i},b_{\ell,i}\}$ are determined by the intersection numbers $\{c_{\bar \ell,i},b_{\bar \ell,i}\}$. In particular both sequences are the same iff $k_0=k_1$ and in this case $\Gamma$ is a (bipartite) distance-regular graph.

\begin{lemma} \cite[Corollary 2.11]{A1990}\label{lem:DBRGdiametercases}
Let $\Gamma$ be a distance-bipartite regular graph. We can assume w.l.o.g. that one of the following holds
\begin{enumerate}
    \item $D_0=D_{1}$ and $k_0=k_{1}$; so $\Gamma$ is a bipartite distance-regular graph.
    \item $D_0=D_1-1$ is odd and $k_0>k_1$.
    \item $D_0=D_{1}$ is even and $k_0>k_{1}$.
\end{enumerate} 
\end{lemma}

 We display some preliminary results about the intersection parameters of a distance-biregular graphs, whose proofs are omitted since they follow trivially from \cite{A1990}.

\begin{lemma}
If $k_0>k_{1}$, then  
$$\begin{array}{rl}
\dfrac{b_{1, i}}{b_{0, i}}< \dfrac{k_{1}}{k_0}< \dfrac{c_{1,i}}{c_{0,i}}, &  \mbox{ if } i \mbox{ is even},\\[2ex]
\dfrac{b_{0, i}}{b_{1, i}}< \dfrac{k_{1}}{k_0}< \dfrac{c_{0,i}}{c_{1,i}}, &  \mbox{ if } i \mbox{ is odd}.
\end{array}$$
\end{lemma}

The result provides an explicit expression of the equilibrium measure for sets $V\setminus\{y\}, \forall y\in V$ of distance-biregular graphs.
\begin{proposition}
\label{eq_distancebiregular}
  Let $\Gamma$ be a distance-biregular graph with $V=V_0 \cup V_1$. Then, for any $\ell=0,1$, there exists an array $q_{\ell}$ of length $D_{\ell}$ such that if $x\in V_\ell$, for any $y\in V$ it holds
  $$\nu^{x}(y)=q_{\ell,m}\,\, \Longleftrightarrow\,\, d(x,y)=m, \,\, m=0,\ldots,D_{\ell}.$$ 
  Moreover, 
  $$q_{\ell,m}=\displaystyle\sum\limits_{j=0}^{m-1}\dfrac{n-B_{\ell,j}}{k_{\ell,j}b_{\ell,j}}=\sum\limits_{j=1}^{m}\dfrac{n-B_{\ell,j-1}}{k_{\ell,j}c_{\ell,j}}.$$
In particular,
$q_{\ell,1}=\dfrac{n-1}{k_{\ell}}$ and $q_{\ell,2}=\dfrac{n-1}{k_{\ell}}+\dfrac{n-1-k_{\ell}}{k_{\ell}(k_{\bar\ell}-1)}.$
     \end{proposition}

\begin{proof}
Take $x\in V_\ell$ with $\ell=0,1$. Assume that the value $\nu^x(y)$ depends only on the distance from $x$ to $y$, that is, there exists $q_{\ell,i}$, $i=1,\ldots,D_{\ell}$ such that $\nu^{x}(y)=q_{\ell,i}\,\, \Longleftrightarrow\,\, d(x,y)=i.$ Moreover, we define $q_{_{\ell,D_{\ell}+1}}=0$. Note that, since the equilibrium system $\L\nu^x(y)=1$ for all $y\in V\setminus\{x\}$ has a unique solution, then if with our hypothesis we can solve the system, such solution must correspond to the equilibrium measure $\nu^x(y)=q_{\ell,i}$.

In our case, $\L\nu^x(y)=1$ for all $y\in V\setminus\{x\}$ is equivalent to the  system
$$ (b_{\ell,i}+c_{\ell,i})q_{\ell,i}-c_{\ell,i}q_{\ell,i-1}-b_{\ell,i}q_{\ell,i+1}=1,\,\,\,\, i=1,\ldots, D_{\ell}$$
where 
$\ell=0,1$. Multiplying by $k_{\ell,i}$, we obtain
$$ k_{\ell,i} c_{\ell,i}(q_{\ell,i}-q_{\ell,i-1}) - k_{\ell,i}b_{\ell,i}(q_{\ell,i+1}-q_{\ell,i}) = k_{\ell,i},\,\,\,\, i=1,\ldots, D_{\ell}.$$

Since $k_{\ell,i}c_{\ell,i}=k_{\ell,i-1}b_{\ell,i-1}$ and denoting  $\gamma_{\ell,i}=k_{\ell,i}b_{\ell,i}(q_{\ell,i+1}-q_{\ell,i})$, then
$$\gamma_{\ell,{i-1}} - \gamma_{\ell,i} = k_{\ell,i}, \qquad \text{for } i=1,\ldots, D_{\ell}.$$

Observing that $\gamma_{_{\ell,D_{\ell}}}=0$, then summing up
$$n-B_{\ell,j}=\sum\limits_{i=j+1}^{D_{\ell}}k_{\ell,i}=\gamma_{\ell,j}-\gamma_{_{\ell,D_{\ell}}}=\gamma_{\ell,j}$$
for $j=0,\ldots, D_{\ell}-1$, it follows
$$q_{\ell,i+1}-q_{\ell,i}=\frac{n-B_{\ell,i}}{k_{\ell,i}b_{\ell,i}}, \qquad \text{for }i=0,\ldots, D_{\ell}-1.$$

Finally, since $q_{\ell,0}=0$, it follows
$$q_{\ell,m}=\sum\limits_{j=0}^{m-1} \frac{n-B_{\ell,j}}{k_{\ell,j}b_{\ell,j}}, \qquad \text{for }m=0,\ldots, D_{\ell}.$$

The expression for $q_{\ell,m}$ in terms of $c_{\ell,i}$ follows from Lemma \ref{propertiess} $(ii)$.
\end{proof}

The above proposition motives the definition of equilibrium arrays for a distance-biregular graph. If $\Gamma$ is a distance-biregular graph, we call \emph{equilbrium arrays} to the values $q_{\ell,i},\,\, \ell=0,1$ and $i=0,\ldots,D_{\ell}$. Denote  $m_{\ell,i}=\big|\{y\in V: \nu^x(y)=q_{\ell,i}\}\big|.$

\begin{cor}\label{equilibrium-relation}
 Let $\Gamma$ be a distance-biregular graph with $y\in V_\ell$ and $x\in V_{\hat\ell}$, $\ell,\hat\ell=0,1$. Then,  %
 $$q_{\ell,d(x,y)}=q_{\hat\ell,d(x,y)}+(n-1)\left(\dfrac{1}{k_{\ell}}-\dfrac{1}{k_{\hat\ell}}\right).$$
\end{cor}

\begin{proof}
 From \eqref{equil-sym}  we know that 
$\nu^y(x)-\nu^x(y)=\frac{1}{n}\big(\ccap(y)-\ccap(x)\big)$. On the other hand, $\ccap(y)=\ccap(z)$ for any $z\in V_{\ell}$ and $\ccap(x)=\ccap(w)$ for any $w\in V_{\hat\ell}.$ So, for $\ell\not=\hat\ell$, we can choose $z\in V_{\ell}$ and $w\in V_{\hat\ell}$ such that $d(z,w)=1$, then 
$$\frac{1}{n}\big(\ccap(y)-\ccap(x)\big)=\frac{1}{n}\big(\ccap(z)-\ccap(w)\big)=q_{\ell,1}-q_{\hat\ell,1}=(n-1)\left(\dfrac{1}{k_{\ell}}-\dfrac{1}{k_{\hat\ell}}\right),$$
which implies that 
$$q_{\ell,d(x,y)}=q_{\hat\ell,d(x,y)}+(n-1)\left(\dfrac{1}{k_{\ell}}-\dfrac{1}{k_{\hat\ell}}\right).$$
If $\ell=\hat\ell$, the result trivially holds.
\end{proof}

As an straightforward application of Proposition \ref{eq_distancebiregular} we can find the intersection array of a distance-biregular graph in terms of the equilibrium arrays, analogously as it was done in \cite[Proposition 4.5]{BCEM2000shortest} for distance-regular graphs.

\begin{proposition}\label{prop:arrayintermspotential}
  Let $\Gamma$ be a distance-biregular graph with equilibrium arrays $q_{\ell,i}$ for $\ell=0,1$ and $i=0,\ldots,D_\ell$. Then, for any $i=0,\ldots,D_{\ell}-1$, it holds
\begin{align*}
   &k_{\ell,i}=m_{\ell,i}, \\
   &b_{\ell,i}=\frac{1}{m_{\ell,i}(q_{\ell,i+1}-q_{\ell,i})}\sum_{j=i+1}^{D_{\ell}}m_{\ell,j},\\ &c_{\ell,i+1}=\frac{1}{m_{\ell,i+1}(q_{\ell,i+1}-q_{\ell,i})}\sum_{j=i+1}^{D_{\ell}}m_{\ell,j}.  
\end{align*}

\end{proposition}

The computation of the equilibrium measure is usually done using linear programming \cite{BCEM2000shortest}. In this regard, Proposition \ref{prop:arrayintermspotential} provides a tool to calculate the  intersection arrays of a distance-biregular graph solving one linear system.


Another application of the equilibrium measure concerns the estimation of the effective resistance
of a resistive electrical network.  As a consequence of Proposition \ref{eq_distancebiregular}, we show the the effective resistance of distance-biregular graphs.

\begin{cor}
 Let $\Gamma$ be a distance-biregular graph with $y\in V_\ell$ and $x\in V_{\hat\ell}$, $\ell,\hat\ell=0,1$. Then, the effective resistance bewteen $x,y$  is
 $$R(x,y)=\frac{2}{n}q_{\ell,d(x,y)}+\frac{(n-1)}{n}\left(\dfrac{1}{k_{\hat\ell}}-\dfrac{1}{k_{\ell}}\right).$$
\end{cor}

\begin{proof}
 The result follows from Corollary \ref{equilibrium-relation} taking into account that  
$R(x,y)=\dfrac{\nu^{x}(y)+\nu^{y}(x)}{n}$, see \cite{BCEM03}. 
\end{proof}

The next results shows the group inverse of the Laplacian of a distance-biregular graphs in terms of the intersection arrays. 

\begin{thm}\label{greenfunctiondistancebiregular}
Let $\Gamma$ be a distance-biregular graph. Then, for each $y\in V_\ell$ with $\ell=0,1$, the group inverse of $\LL$ is given by


$$\LL^\#(x,y)=
\displaystyle\displaystyle\dfrac{1}{n}\sum\limits_{j=d(x,y)+1}^{D_\ell} \frac{n-B_{\ell,j-1}}{k_{\ell,j}c_{\ell,j}} -\displaystyle\dfrac{1}{n^2}\sum\limits_{j=1}^{D_\ell}
\frac{B_{\ell,j-1}(n-B_{\ell,j-1})}{k_{\ell,j}c_{\ell,j}} .$$
\end{thm}

\begin{proof}

From 
\eqref{green-f}, we know that $\LL^\#(x,y)=\frac{1}{n^2}(\ccap(y)-n\nu^y(x)).$ Take $y\in V_\ell$ with $\ell=0,1$. Now, using Proposition \ref{eq_distancebiregular}, 
\begin{align*}
    \ccap(y)&=\displaystyle\sum\limits_{x\in V}\nu^y(x)=\displaystyle\sum\limits_{m=0}^{D_{\ell}} k_{\ell,m}q_{\ell,m}=\sum\limits_{m=0}^{D_{\ell}} k_{\ell,m}\sum\limits_{j=1}^{m}\dfrac{n-B_{\ell,j-1}}{k_{\ell,j}c_{\ell,j}}\\
    &= \displaystyle\sum\limits_{j=1}^{D_{\ell}}\sum\limits_{m=j}^{D_\ell}   \frac{k_{\ell,m}(n-B_{\ell,j-1})}{k_{\ell,j}c_{\ell,j}}= \displaystyle\sum\limits_{j=1}^{D_{\ell}} \frac{(n-B_{\ell,j-1})^2}{k_{\ell,j}c_{\ell,j}}.
\end{align*}
Therefore,
\begin{align*}
    \LL^\#(x,y)=&\displaystyle\sum\limits_{j=1}^{D_{\ell}} \frac{(n-B_{\ell,j-1})^2}{n^2k_{\ell,j}c_{\ell,j}}  - \sum\limits_{j=1}^{d(x,y)}   \frac{n-B_{\ell,j-1}}{nk_{\ell,j}c_{\ell,j}}\\[3ex]
    =&\displaystyle\sum\limits_{j=d(x,y)+1}^{D_\ell} \frac{n-B_{\ell,j-1}}{nk_{\ell,j}c_{\ell,j}}-
\sum\limits_{j=1}^{D_\ell}\frac{B_{\ell,j-1}(n-B_{\ell,j-1})}{n^2k_{\ell,j}c_{\ell,j}}. 
\end{align*}

\end{proof}

\begin{remark}
Observe that, since the  intersection numbers of a distance-biregular graph are related, from Theorem \ref{greenfunctiondistancebiregular}, the  expression of the group inverse is equivalent to
$$\LL^\#(x,y)=\displaystyle\dfrac{1}{n}\sum\limits_{j=d(x,y)}^{D_\ell-1}\frac{n-B_{\ell,j}}{k_{\ell,j}b_{\ell,j}} -\displaystyle\dfrac{1}{n^2}
\sum\limits_{j=0}^{D_{\ell}-1}\frac{B_{\ell,j}(n-B_{\ell,j})}{k_{\ell,j}b_{\ell,j}}.$$
\end{remark}

\begin{example}\label{star-and-all}
As an application of Theorem \ref{greenfunctiondistancebiregular} we obtain the group inverse of the Laplacian of a complete bipartite graph using its parameters: 
$$\left.\begin{array}{ll}
D_0=2, k_0, c_{0,1}=1, c_{0,2}=k_0, b_{0,0}=k_0, b_{0,1}=k_1-1, 
\\[1ex]
D_1=2, k_1, c_{1,1}=1, c_{1,2}=k_1, b_{1,0}=k_1, b_{1,1}=k_0-1,
%
\end{array}\right\}\Longrightarrow n=k_0+k_1.$$
Take $x,\hat x\in V_0$, $\hat x\not=x$ and $y,\hat y\in V_1$, $\hat y\not=y$. Then,
$$\begin{array}{rl}
{\sf L}^\#(x,x)=&\hspace{-.25cm}\dfrac{1}{n^2}\bigg[  \dfrac{(k_0+k_1-1)^2}{k_0} + \dfrac{k_1-1}{k_0} \bigg]=\dfrac{(n-1)^2+n-k_0-1}{n^2k_0}=\dfrac{n^2-n-k_0}{k_0n^2},\\[3ex]
{\sf L}^\#(y,y)=&\hspace{-.25cm}\dfrac{1}{n^2}\bigg[  \dfrac{(k_0+k_1-1)^2}{k_1}  + \dfrac{k_0-1}{k_1} \bigg]=\dfrac{n^2-n-k_1}{k_1n^2},\\[3ex]
{\sf L}^\#(y,x)=&\hspace{-.25cm}\dfrac{n^2-n-k_0}{k_0n^2}-\dfrac{n(n-1)}{k_0n^2}=-\dfrac{1}{n^2}={\sf L}^\#(x,y),\\[3ex]
{\sf L}^\#(\hat x,x)=&\hspace{-.25cm}\dfrac{n^2-n-k_0}{k_0n^2}-\dfrac{n(n-1)}{k_0n^2}-\dfrac{n}{k_0n^2}=-\dfrac{(n+k_0)}{k_0n^2},\\[3ex]
{\sf L}^\#(\hat y,y)=&\hspace{-.25cm}-\dfrac{(n+k_1)}{k_1n^2}.
 \end{array}$$

Observe that for a complete bipartite graph, it holds that ${\sf L}^\#$ is always an $M$-matrix. The above expression is valid when $D_0=D_1=1$, and $D_0=1$ and $D_1=2$; that is, for the star graph. 

\end{example}

\section{Distance-biregular graphs with the $M$-property}

In this section, we answer Question \ref{que:Mmatrix} for distance-biregular graphs, completing, together with the known results for distance-regular graphs \cite{BCEM2012}, the characterization of when the group  inverse of the combinatorial Laplacian matrix of a distance-regularised  graph is an $M$-matrix.

\begin{proposition}\label{Mpropertydistancebiregular}
Let $\Gamma$ be a distance-biregular graph. Then, $\Gamma$ has the $M$-property if and only if,  it holds
\begin{align*}
  & 
  \displaystyle\sum\limits_{j=1}^{D_{0}-1} \frac{1}{k_{0,j}b_{0,j}} \Big (\sum\limits_{i=j+1}^{D_{0}}k_{0,i}\Big)^2 \leq  \frac{(n-1)}{k_0}.
\end{align*}
\end{proposition}

\begin{proof}
We know that a graph $\Gamma$ satisfies the $M$-property if and only if the entries of ${\sf L}^\#(x,y)\leq 0$ for all $x\sim y$. The result follows from using that 
$$k_{0,0}=1,\,\,
    b_{0,0}=k_0,\,\,
 \mbox{ and }\,\, \sum\limits_{i=1}^{D_{0}}k_{0,i}=n-1.
$$

\end{proof}

\begin{remark}
The condition from Proposition \ref{Mpropertydistancebiregular} is equivalent to
\begin{align}\label{thm:Mpropertyconb}
  & 
  \displaystyle\sum\limits_{j=1}^{D_{1}-1} \frac{1}{k_{1,j}b_{1,j}} \Big (\sum\limits_{i=j+1}^{D_{1}}k_{1,i}\Big)^2 \leq  \frac{(n-1)}{k_1}.
\end{align}
\end{remark}

Using Proposition \ref{Mpropertydistancebiregular} we can also obtain the following necessary condition for a distance-biregular graph having the $M$-property.

\begin{cor}
\label{penrose-distance-regular:necessary} If $\Gamma$ is a distance-biregular graph with  the $M$-property and $D_0 \ge 2$, then
$$ n< 2k_1+k_{0}.$$
 \end{cor}

\begin{proof}
\, Since $D_1\ge D_0\ge 2$, from Proposition
\ref{Mpropertydistancebiregular} we obtain that

\begin{align*}
\frac{1}{k_{1,1}b_{1,1}} \Big (\sum\limits_{i=2}^{D_{1}}k_{1,i}\Big)^2  \leq \displaystyle\sum\limits_{j=1}^{D_{1}-1} \frac{1}{k_{1,j}b_{1,j}} \Big (\sum\limits_{i=j+1}^{D_{1}}k_{1,i}\Big)^2 \leq  \frac{(n-1)}{k_1}.
\end{align*}

Now, observing that
$$\dfrac{1}{k_{1,1}b_{1,1}} \Big (\sum\limits_{i=2}^{D_{1}}k_{1,i}\Big)^2=\dfrac{(n-k_1-1)^2}{k_{1}b_{1,1}}$$ we get

\begin{align*}
    &(n-k_1-1)^2 \leq (n-1)b_{1,1}=(n-1)(k_{0}-1) \Longleftrightarrow\\
    & (n-1)^2 -2k_1(n-1)+k_1^2\leq (n-1)(k_{0}-1) \Longleftrightarrow\\
    &n-2k_{1}+\dfrac{k_{1}^2}{n-1}-k_{0}\le 0,
\end{align*}

and since $\frac{k_1^2}{n-1} > 0$, the result follows.
\end{proof}

Note that the inequality $n<2k_1+k_{0}$ turns out to be a strong restriction for a
distance-biregular graph to have the $M$-property. Observe that such condition implies that the distance-biregular graph needs to be quite dense.

%
%
%


The following result generalizes the above observation  by showing that only distance-biregular graphs with small $D_\ell$ can satisfy the $M$-property. A related result appeared in \cite[Proposition 5]{BCEM2012}, where it was shown that the diameter of a distance-regular graphs with the $M$-property must be at most $3$.

\begin{proposition}\label{DBRGMpropertydiameter}
If $\Gamma$ is a distance-biregular graph with the $M$-property, then $D_{1}\le 4$ and $D_{0}\le 3$.
\end{proposition}
\begin{proof}
\, By means of a contradiction, assume $D_0,D_{1}\ge 4$. Then, 
$$  1+k_\ell+k_{\ell,2}+k_{\ell,3}< 1+k_\ell+k_{\ell,2}+k_{\ell,3}+k_{\ell,4}\le n.$$
We can assume that $k_0>k_{1},$ since otherwise $\Gamma$ is a bipartite distance-regular graphs and hence $D_0=D_1\le 3$, see \cite[Proposition 5]{BCEM2012}. Then,  

$$k_{0,2}=k_{0}\dfrac{b_{0,1}}{c_{0,2}}\ge k_{0}\dfrac{c_{1,2}}{c_{0,2}}>k_{0}\dfrac{k_{1}}{k_{0}}=k_{1}.$$
On the other hand, since $b_{0,1}\ge c_{0,3}$ and $b_{0,2}\ge c_{0,2}$, we obtain that 
$$k_{0,3}=k_0 \dfrac{b_{0,1}b_{0,2}}{c_{0,2}c_{0,3}}\ge k_0.$$

Finally, from Corollary \ref{penrose-distance-regular:necessary} it follows that $n+1<1+2k_0+k_{1}\le n$, a contradiction. 
\end{proof}


As an application of Proposition \ref{DBRGMpropertydiameter}, we classify distance-biregular graphs having the $M$-property. We follow the notation from \cite{A1990}.

\medskip
\begin{description}
\item[Case 1:] $D_0=D_1=1$. This case corresponds to a digon  that is a distance-regular and has the $M$-property.

\item[Case 2:] $D_0=1$, $D_1=2$. This case corresponds to star graphs, which are known to have the $M$-property, see \cite{CEM2014, KN98} or Example \ref{star-and-all}.


\medskip
\item[Case 3:] $D_0=D_1=2$. This case corresponds to a bipartite complete graph (see Example \ref{star-and-all}).


\medskip
\item[Case 4:] $D_0=D_1=3$. This case corresponds to a bipartite distance-regular graph \cite[Section 5.1]{A1990}, and thus it was already studied in \cite{BCEM2012}. In this case, the intersection array is $\{k,k-1,k-\mu; 1,\mu,k\}$, where $1\le
\mu\le k-1$ and  $\mu$ divides  $k(k-1)$. They are antipodal iff
$\mu=k-1$. Otherwise, they are the incidence graphs of nontrivial
square $2-\Big(\frac{n}{2},k,\mu\Big)$ designs. Therefore, $k-\mu$
must be a square, see \cite[Th. 1.10.4]{BCN89}.

\begin{proposition}\cite[Proposition 13]{BCEM2012}
  \label{bipartite} A bipartite distance-regular graph with $D=3$ satisfies
  the $M$-property if and only if $$\dfrac{4k}{5}\le \mu\le k-1$$ and these inequalities imply that $k\ge
  5$. In particular, if $1\le\mu< k-1$, then either $\Gamma$ or $\Gamma_3$ has the $M$-property, except when
$  k-1<5\mu<4k,$
in which case none of them has
the $M$-property.
\end{proposition}


\medskip
\item[Case 5:] $D_0=3,D_1=4$ with $k_0>k_1$. In \cite[Proposition 5.3]{A1990} it is shown that $\Gamma$ is the point-line incidence graph of a quasi-symmetric design with $x=0$  if and only if $\Gamma$ is a distance-biregular graph with $D_0=3,D_1=4$ and intersection array 
\begin{align*}
\left\{ \begin{array}{lllll} 
    r; & 1, & \lambda, & k & \\
    k; & 1, & y , & \frac{k\lambda}{y}, & k \\
                \end{array} \right\}.
\end{align*}
Recall that a $2$-$(v,k,\lambda)$ \emph{quasi-symmetric design} is a design with two intersection numbers, and we are interested in those having $x=0<y<k.$ Moreover, $k\lambda$ needs to be a multiple of $y$, that is, $k \lambda=\alpha y$, $\alpha\in\mathbb{N}$. Also, recall that $r> \lambda$. Moreover, since $B_{0,D_0}=B_{1,D_1}$, it holds that $(y-1)(r-1)=(k-1)(\lambda-1)$, see also \cite{BS82} for a proof based on design techniques.

\end{description}

Next, using the condition in Proposition \ref{Mpropertydistancebiregular} with $r=k_0, k=k_1$, we obtain  a necessary and sufficient condition for a distance-biregular graph with $D_0=3,\,\,D_1=4$ to have the $M$-property.

\begin{proposition}\label{prop:conditionMproperty}
  A distance-biregular graph with diameters $D_0=3,D_1=4$ has the $M$-property if and only if
$$\displaystyle (k-1)(r-\lambda)\Big( (k+r)^2-\lambda k\Big) \leq  k^2\lambda^2.
$$

\end{proposition}

\begin{proof}
We use   Proposition \ref{Mpropertydistancebiregular} to obtain:

\begin{align*}
  & 
  \displaystyle \frac{1}{k_{0,1}b_{0,1}} \Big (k_{0,2}+k_{0,3}\Big)^2+\frac{1}{k_{0,2}b_{0,2}} \Big (k_{0,3}\Big)^2 \leq  \frac{(n-1)}{k_0}.
\end{align*}
Keeping in mind that 
$$\begin{array}{llll}
k_0=r, &  b_{0,1}=k-1, & b_{0,2}=r-\lambda, &\\[3ex]
k_{0,0}=1, & k_{0,1}=r, & k_{0,2}=r\dfrac{(k-1)}{\lambda}, &\,\, k_{0,3}=r\dfrac{(k-1)}{\lambda}\dfrac{(r-\lambda)}{k} ,
\end{array}
$$
we get that
$$
\dfrac{n-1}{r}=1+\dfrac{(k-1)}{\lambda k}\big(k+r-\lambda\big).$$
After performing some simplifications on the first inequality, the desired result follows.
\end{proof}

Finally, we study the $M$-property for some classes of distance-biregular graphs with $D_0=3$ and $D_1=4$. 

\begin{example} Consider the point-line incident graph of the affine plane  ${\cal A}(2,n)$ of order $n$, whose intersection array is
\begin{align*}
\left\{ \begin{array}{lllll} 
    n+1; & 1, & 1, & n & \\
    n; & 1, & 1 , & n, & n \\
                \end{array} \right\}.
\end{align*}
It is easy to check that it does not verify the inequality in Proposition \ref{Mpropertydistancebiregular} and hence it does not verify the $M$-property.
\end{example}

\begin{example} Let $\Gamma=S(K_{r+1})$ be the subdivision graph of the the complete graph $K_{r+1}$. This is a class of distance-biregular graphs with diameters $D_0=4$, $D_1=3$ and  parameters   

{\footnotesize{
$$\begin{array}{rl}
 D_0=3,\,\,   k_0=r, & c_{0,1}=c_{0,2}=1, c_{0,3}=2, b_{0,0}=r, b_{0,1}=1, b_{0,2}=r-1;  \\[2ex]
 D_1=4,\,\,  k_1=2, & c_{1,1}=c_{1,2}=1, c_{1,3}=c_{1,4}=2, b_{1,0}=2, b_{1,1}=r-1, b_{1,2}=1, b_{1,3}=r-2.
\end{array}$$
}}
which does not hold the condition from Proposition \ref{Mpropertydistancebiregular}, and thus does not have the $M$-property. In fact, we can find the group inverse. We denote ${\sf L}^\#_{\ell,j}={\sf L}^\#(x,y)$ when $y\in V_\ell$ and $d(x,y)=j$. Then,

{\footnotesize{
 $$\begin{array}{rlrlrl}
{\sf L}^\#_{0,0}=&\hspace{-.25cm}\dfrac{r(r + 3)(2r + 3)}{(r + 1)^2(r + 2)^2}, & {\sf L}^\#_{0,1}=&\hspace{-.25cm}\dfrac{(r + 3)(r^2 - 2)}{(r + 1)^2(r + 2)^2}, & {\sf L}^\#_{0,2}=&\hspace{-.25cm} -\dfrac{r^2 + 7r + 8}{(r + 1)^2(r + 2)^2}, \\[3ex]
{\sf L}^\#_{0,3}=&\hspace{-.25cm} -\dfrac{2(r^2 + 5r + 5)}{(r + 1)^2(r + 2)^2}, & & & &  \end{array}$$
}}
and 

{\footnotesize{
$$\begin{array}{rlrlrl}
{\sf L}^\#_{1,0}=&\hspace{-.25cm}\dfrac{(r + 3)(r^3 + 5r^2 + 2r - 4)}{2(r + 1)^2(r + 2)^2}, & {\sf L}^\#_{1,1}=&\hspace{-.25cm}\dfrac{(r + 3)(r^2 - 2)}{(r + 1)^2(r + 2)^2}, & {\sf L}^\#_{1,2}=& \hspace{-.25cm}\dfrac{r^3 - r^2 - 18r - 20}{2(r + 1)^2(r + 2)^2}, \\[3ex]
{\sf L}^\#_{1,3}=&\hspace{-.25cm} -\dfrac{2(r^2 + 5r + 5)}{(r + 1)^2(r + 2)^2}, & {\sf L}^\#_{1,4}=&\hspace{-.25cm} -\dfrac{(r + 3)(3r + 4)}{(r + 1)^2(r + 2)^2}. & &  \end{array}$$
}}
\end{example}

For the classification of existing quasi-symmetric 2-designs, see \cite[Table 48.25]{Shrikhande}. We should note that for the existing quasi-symmetric 2-designs with $x=0$, none passes the condition from Proposition  \ref{prop:conditionMproperty}.

The above discussion extends the results in \cite{BCEM2012} and completes the classification of distance-regularised graphs that have the $M$-property. There, it was shown that if there are distance-regular graphs with valency $k \geq 3$ and diameter $D \geq 2$ having the $M$-property, then they have at most $3k$ vertices and $D\leq 3$. Also in \cite{BCEM2012}, it was conjectured that there is no primitive distance-regular graph with diameter $3$ having the $M$-property. This conjecture was shown to be true except
possibly for finitely many primitive distance-regular graphs \cite[Theorem 1]{KP2013}.

In view of the above results, we conclude this paper with the following conjecture.

\begin{con}
There are not point-line incidence graphs of a $2$-quasi-symmetric design with $x=0$ that have the $M$-property.
\end{con}

\section*{Acknowledgments}
The research of A. Abiad is partially supported by the FWO grant 1285921N.

\end{document}